\documentclass{amsart}
\usepackage{graphicx}
\usepackage{color}
\usepackage{amssymb}
\usepackage{hyperref}
\usepackage[utf8]{inputenc}
\usepackage[T1]{fontenc}
\usepackage[all]{xy}
\newtheorem{thm}{Theorem}[section]
\newtheorem{cor}[thm]{Corollary}
\newtheorem{lem}[thm]{Lemma}
\newtheorem{prop}[thm]{Proposition}
\theoremstyle{definition}
\newtheorem{defn}[thm]{Definition}
\theoremstyle{remark}

\numberwithin{equation}{section}

\newcommand{\rs}{Z_P} 

\newcommand{\C}{\mathbb{C}}
\newcommand{\R}{\mathbb{R}}
\newcommand{\Q}{\mathbb{Q}}
\newcommand{\Z}{\mathbb{Z}}

\newcommand{\X}{{\mathcal {X}}}

\newcommand{\D}{{\mathcal {D}}}

\newcommand{\Rv}{\mathcal{R}_3}
\newcommand{\Rp}{\mathcal{R}_{24}}
\newcommand{\rv}{r_3}
\newcommand{\rp}{r_{24}}

\newcommand{\rvprim}{r_{3, prim}}

\newcommand{\adot}{\!\cdot\!} 

\DeclareMathOperator{\disc}{disc}
\DeclareMathOperator{\gr}{Gr_{2,4}(\Z)}
\DeclareMathOperator{\PGL}{PGL_{2}}
\DeclareMathOperator{\PO}{PO(2)}
\DeclareMathOperator{\vol}{vol}
\DeclareMathOperator{\GL}{GL}
\DeclareMathOperator{\cl}{\mathcal C\ell}
\DeclareMathOperator{\Sp}{Sp}

\newcommand{\id}{I_4}

\DeclareMathOperator{\diag}{diag}
\DeclareMathOperator{\re}{Re}
\DeclareMathOperator{\gen}{gen}
\DeclareMathOperator{\rk}{rk}

\begin{document}

\title{Planes in $\Z^4$ and Eisenstein series}
\author{Gautam Chinta}%
\address{Department of Mathematics, The City College of New York, New York, NY 10031}
\email{gchinta@ccny.cuny.edu}
\author{Valdir Pereira J\'{u}nior}%
\address{Department of Mathematics, The City College of New York, New York, NY 10031}
\email{valdirjr@impa.br}%

\subjclass{Number theory}%


\begin{abstract}  We study the number of two-dimensional sublattices of $\Z^4$ of a fixed covolume and construct the associated Dirichlet series.  The latter is shown to be related to Eisenstein series on both $\GL_4$ and its metaplectic double cover.
\end{abstract}
\maketitle


\section{Introduction}

The study of two dimensional sublattices in the integer lattice $\Z^4$ has many parallels with the study of the sums of three integer squares.  The result of Legendre and Gauss states that a positive integer $n$ can be written as a sum of three squares if and only if $n$ is not of the form $4^a(8b-1)$ for $a,b$ nonnegative integers, i.e. if $4\nmid n$, then $n$ is a sum of three squares precisely when $n\not\equiv 7 \mod 8$.  Gauss shows further that the number of representations of $n$ as a sum of three integer squares is related to the class number of binary quadratic forms of discriminant $-4n$.

Also of interest is the distribution of these representations projected onto the unit sphere as $n\to\infty.$  That is, let
\begin{equation}
  \label{eq:1}
\Rv(n)=\{(x,y,z)\in\Z^3: x^2+y^2+z^2=n\}
\end{equation}
be the set of
representations of $n$ as a sum of three squares and let
\begin{equation}
  \label{eq:3}
\rv(n)=\#\Rv(n).
\end{equation}
In influential work begun in the 1950's, Linnik considers the distribution of the points $\{v/\sqrt{n}:v\in\Rv(n)\}$ and shows that these points become equidistributed on the unit sphere as $n\to\infty$, under suitable congruence congruence conditions, see \cite{Li}.  The proof introduces what is now called \emph{Linnik's ergodic method}.  Using completely different techniques from the theory of automorphic forms, Duke \cite{du} gave a new proof of Linnik's result and further removed Linnik's congruence assumption.  Duke's proof relies on work of Iwaniec \cite{iw} on nontrivial estimates for Fourier coefficients of half-integral weight modular forms.

The coefficients $\rv(n)$ arise naturally in the study of half-integral weight modular forms and, more generally, in the study of Eisenstein series on metaplectic covering groups.  First, these coefficients are the Fourier coefficients of a classical weight 3/2 theta series.  The associated Dirichlet series occurs as a Fourier-Whittaker coefficient of a metaplectic Eisenstein series on the metaplectic double cover on $\GL(3)$, see e.g. \cite{CO2}.  A generalization of this last statement is the main result of this paper.

Now consider the four dimensional integer lattice $\Z^4$ equipped with the standard inner product and associated quadratic form $Q(x_1,x_2,x_3,x_4)=x_1^2+x_2^2+x_3^2+x_4^2.$  The \emph{discriminant} of a two dimensional sublattice $L$ of $\Z^4$ is defined to be the discriminant of the binary quadratic form $Q|_L$ obtained by restricting $Q$ to $L$.  Equivalently, the discriminant of $L$ is $-4$ times the square of the volume of the quotient $W/L$, where $W$ is the two dimensional subspace of $\R^4$ containing $L$.  The lattice $L$ is called \emph{primitive} if $L=\Q\cdot L\cap \Z^4,$ i.e. $L$ is generated by some vectors of a basis of $\Z^4.$
Let $\gr$ be the set of primitive two dimensional sublattices of $\Z^4.$  Define
\begin{equation}
  \label{eq:2}
\Rp(n)=\{L\in \gr:\disc(L)=-4n \}.
\end{equation}
and
\begin{equation}
  \label{eq:4}
\rp(n)=\#\Rp(n).
\end{equation}
The paper of Aka, Einsiedler and Wieser \cite{AEW19} presents several results concerning the set $\Rp(n)$ analogous to those of Legendre, Gauss and Linnik about points on the sphere. Among other results, they prove
\begin{itemize}
\item $\Rp(n)$ is nonempty precisely when $n$ is in $n\not\equiv 0,7,12$ or $15$ mod $16.$
\item $\rp(n)$ is related to $\rv(n)^2$, and hence, to the square of the class number of binary quadratic forms of discriminant $-4n$
\item To each $L\in\Rp(n)$ they naturally associate four CM-points $z_1,z_2,z_3,z_4$ on the surface $\X:=\PGL(\Z)\backslash \PGL(\R)/\PO$ and conjecture that the set $\{(L, z_1,z_2,z_3, z_4)|L\in \Rp(n)\}$ becomes equidistributed with respect to the natural uniform measure on $\gr\times\X^4$ as $n\to\infty, n\in \D.$  They prove the equidistribution under certain congruence conditions and on average.
\end{itemize}
We give the definition of the four associated CM-points in Section \ref{sec:4cm}.

In the present paper, we add to this list of analogous results by showing that the representation numbers $\rp(n)$ show up in the Fourier-Whittaker expansion of a metaplectic Eisenstein series on the metaplectic double cover of $\GL_4.$  On the one hand, it is clear that the degenerate \emph{nonmetaplectic} $\GL_4(\Z)$ Eisenstein series associated to the $(2,2)$ parabolic evaluated at the identity matrix can be expressed in terms of the Dirichlet series
\begin{equation}
  \label{eq:5}
\sum_n \frac{\rp(n)}{n^s}
\end{equation}
since this Eisenstein series is a sum over $\gr$.  On the other hand, expressing the Whittaker function of a metaplectic Eisenstein series in terms of (\ref{eq:5}) is more involved and is presented in Section \ref{sec:wmds} in the course of the proof of our main result Theorem \ref{thm:PQ}.

This result falls under a conjectural framework developed by Jacquet  \cite{Jacq91}, who conjectured a relative trace formula identity relating orthogonal periods of an automorphic form on $\GL_n$ to Whittaker functions of an associated form on the metaplectic double cover of $\GL_n$.
Actually, Jacquet's conjecture is for cuspforms but the same can be conjectured for Eisenstein series, and special cases of this are more amenable to direct computation, as in the $\GL_3$ example described in Chinta-Offen \cite{CO2} and the $\GL_4$ example presented here.  Both these works depend heavily on the explicit descriptions of Whittaker coefficients of metaplectic Eisenstein series given in \cite{bbcfh, bbfh, cg-inv}.

In the course of the proof our main Theorem, we prove several results which complement those of Aka-Einsiedler-Wieser \cite{AEW19} from an arithmetic perspective.  Namely,
\begin{itemize}
\item In Proposition \ref{prop:r} we give an explicit formula for $\rp(n)$.  The main ideas for this are already in Section 2.4 of \cite{AEW19} but the precise expression (\ref{Formulrd}) is needed in our proof Theorem  \ref{thm:PQ}.
\item As $L$ ranges over two dimensional sublattices of $\Z^4$ of squarefree discriminant $n\equiv 1 \pmod {4}$, we show in Theorem \ref{WhenExPairLat} that $L$ and its orthogonal complement $L^\bot$ are related by the condition that
\begin{equation}
  \label{eq:11}
[Q|_L]\cdot [Q|_{L^\bot}]
\end{equation}
lies in a fixed genus class in the class group $\cl(-4n).$  Here, $[Q|_L]$ denotes the class of the binary quadratic form obtained by restricting $Q$ to $L$.
\item  The binary quadratic forms associated to the accidental CM points $z_3$ and $z_4$ are shown to be Legendre compositions of $[Q|_L]$ and $[Q|_{L^\bot}]$ (Theorem \ref{CompOrtLatt}).
\end{itemize}

We conclude this introduction with a brief outline of the paper.  Section \ref{sec:genus} gives the definition of genus classes and how they arise in the computation of an orthogonal period of an automorphic form on $\GL_4$.  Section \ref{sec:eisenstein} introduces the Eisenstein series associated to the $(2,2)$ parabolic of $\GL_4.$  Sections \ref{sec:genus} and \ref{sec:eisenstein} closely follow the presentation in Chinta-Offen \cite{CO1}, where an analogous problem for Eisenstein series over a imagainary quadratic field is studied.  In Section \ref{sec:rep} we show that the Eisenstein series can be written as a Dirichlet series whose coefficients count two dimensional sublattices of $\Z^4$, and in Section \ref{sec:formula} we give an explicit formula for these representation number in a special case.   Important for our presentation in these sections is the \emph{Klein map} introduced in Aka-Einsiedler-Wieser \cite{AEW19}.  We prove our main result in Section \ref{sec:wmds}, equating an orthogonal period of the Eisenstein series with a Fourier-Whittaker function of a metaplectic Eisenstein series.  We make use of the computation of the metaplectic Whittaker coefficients due to Brubaker-Bump-Friedberg and Brubaker-Bump-Friedberg-Hoffstein \cite{bbcfh, bbf-xtal} together with an alternate description of Chinta-Gunnells \cite{cg-inv,cg-jams}.  Section \ref{sec:4cm} gives some arithmetic relationships between the four CM point of \cite{AEW19}, and Section \ref{sec:further} lists open questions and avenues for future study.

\vskip .5in
{\bf Acknowledgments}  The second named author gratefully acknowledges the support of Coordena\c{c}\~ao de Aperfei\c{c}oamento de Pessoal de N\'ivel Superior which funded his visit to The City College of New York from March 1--August 30, 2019.  Both authors were supported by NSF-DMS 1601289 and the first named author acknowledges support from a PSC-CUNY Award, jointly funded by The Professional Staff Congress and The City University of New York.

The authors are also grateful to Prof. Akshay Venkatesh who first brought the work of Aka, Einsiedler and Wieser \cite{AEW19} to the authors' attention.

\section{Genus classes and orthogonal periods}
\label{sec:genus}

In this section we relate the period integral of an automorphic form over an anisotropic orthogonal group to a sum over a genus class.

Let $S=\{\infty\}\cup\{ p \ | \ p \ \hbox{is a prime number}\}$ be the set of places of $\mathbb{Q}$ and $S_f=S-\{\infty\}$ the set of finite places.
For an algebraic variety $G$ defined over $\mathbb{Q}$ and a place $v$ of
$\mathbb{Q}$ we denote $G_v=G(\mathbb{Q}_v)$ and
$G_{\mathbb{A}}=G(\mathbb{A})$.
We consider $G=\GL_4$ as an algebraic group over $\mathbb{Q}$.
Let
\[K=O(4)\prod_p\GL_4(\mathbb{Z}_p)\mbox{ and }
K_f=\prod_p\GL_4(\mathbb{Z}_p),
\]
where $O(4)$ is the orthogonal group in $\GL_4(\R).$  Note that $K$ is a maximal compact subgroup of $G_{\mathbb{A}}$.

Let $$X=\{g\in G: {}^tg=g\}$$
be the space of symmetric matrices in $G$. There is an action of $G$
on $X$ given by $g\adot x=gx{}^tg$.  For $x\in X_{\mathbb{Q}}$,
define the class of $x$ to be
$$[x]=\GL_4(\mathbb{Z})\adot x$$
and denote $x\sim y$ if $y\in [x]$. The genus class of $x$ is defined
as

$$[[x]]=X_{\mathbb{Q}}\cap[(G_{\infty}K_f)\adot x],$$
and we denote by $[[x]]/ \sim$ the set of classes in the genus class of $x$. Let $X_{\infty}^{+}$ be the set of positive definite matrices in $X_{\infty}$. It is well known that if $x\in X_{\mathbb{Q}}\cap X_{\infty}^{+}$ then $[[x]]/\sim$ is a finite set (cf. \cite[Prop. 2.3 and Thm. 5.1]{Bor63}).

By \cite[Prop. 2.2]{Bor63}, we have $G_{\mathbb{A}}=G_{\mathbb{Q}}G_{\infty}K_f$, from which follows that the embedding of $G_{\infty}$ in $G_{\mathbb{A}}$ defines a bijection
\[G_{\mathbb{Q}}\backslash G_{\mathbb{A}}/K\simeq G_{\mathbb{Z}}\backslash G_{\infty}/K_{\infty}=\GL_4(\mathbb{Z})\backslash \GL_4(\mathbb{R})/O(4).\]
The symmetric space $\GL_4(\mathbb{R})/O(4)$ is identified with $X_{\infty}^{+}$ via $g\mapsto g\adot \id$. Thus a function $\phi$ on $G_{\mathbb{Q}}\backslash G_{\mathbb{A}}/K$ can be regarded as a function $\phi^{+}$ on $G_{\mathbb{Z}}\backslash X_{\infty}^{+}$ by setting $\phi^{+}(g\adot e)=\phi(g)$ for $g\in G_{\infty}$.  In the sequel we will drop the superscript and use $\phi$ to denote both the function on $G_{\mathbb{Q}}\backslash G_{\mathbb{A}}/K$ and the one on $G_{\mathbb{Z}}\backslash X_{\infty}^{+}.$

Now let $x\in X_\Q\cap X_\infty^+$  and
$$H^x=\{g\in G: g\adot x=x\}$$
be the orthogonal group associated with $x$.  Since $H$ is anisotropic, the quotient $H_{\mathbb{Q}}^x\backslash H_\mathbb{A}^x$ is compact and the period integral
\[P^H(\phi)=\int_{H_{\mathbb{Q}}^x\backslash H_\mathbb{A}^x}\phi(h\theta)dh
\] is well defined for any continuous function  $\phi$ on $H_{\mathbb{Q}}^x\backslash H_\mathbb{A}^x.$

\begin{lem}
Let $\phi$ be a complex valued function on $G_{\mathbb{Q}}\backslash G_{\mathbb{A}}/K$ and $x\in X_{\mathbb{Q}}\cap X_{\infty}^{+}.$  Choose $\theta\in G_\infty$ satisfying $\theta\adot\id=x.$ we have

$$\int_{H_{\mathbb{Q}}^x\backslash H_\mathbb{A}^x}\phi(h\theta)\,dh=\vol((H_{\mathbb{A}_f}^x\cap K_f)H_{\infty}^x)\sum_{[y]\in[[x]]/\sim}\epsilon(y)^{-1}\phi(y)$$
where
$\epsilon(y)=\#\{g\in G(\mathbb{Z}) \ | \ g\adot y=y\}.$
\end{lem}

\begin{proof}
A proof can be readily adapted from \cite[Lem. 2.1]{CO1}.
\end{proof}
\bigskip


\section{The $(2,2)$ Eisenstein series}
\label{sec:eisenstein}

In this section we introduce the Eisenstein series on $\GL_4$ induced from characters on the standard parabolic subgroup of type $(2,2).$  Let $P$ be the $(2,2)$ parabolic subgroup of $\GL_4$ containing the Borel subgroup of upper triangular matrices.  Write $P=MU$, where $U$ is the unipotent radical and $M$ is the standard Levi subgroup of $P$. For $\mu=(\mu_1,\mu_2)\in \mathbb{C}^2$, such that $\mu_1+\mu_2=0$, we associate the character of $M_{\mathbb{A}}=P_{\mathbb{A}}/U_{\mathbb{A}}$
\begin{equation*}
\diag(m_1,m_2)\mapsto |\det m_1|_{\mathbb{A}}^{\mu_1}|\det m_2|_{\mathbb{A}}^{\mu_2},
\end{equation*}
which we denote also by $\mu$. We denote by $I_P^G(\mu)=Ind_{P_\mathbb{A}}^{G_{\mathbb{A}}}(\mu\cdot \delta_{P_{\mathbb{A}}}^{1/2})$ the normalized induction of $\mu$ from $M_{\mathbb{A}}$ to $G_{\mathbb{A}}$. For $\varphi\in I_P^G(\mu)$, define the Eisenstein series
\begin{equation*}
E_P(g,\varphi,\mu)=\sum_{\gamma\in P_{\mathbb{Q}}\backslash G_{\mathbb{Q}}}\varphi(\gamma g),
\end{equation*}
The sum is absolutely convergent for $\re(\mu_1-\mu_2)$ sufficiently large, then meromorphically continued to $\mu_1\in \C.$

For  $m=\diag(m_1,m_2)\in M_{\mathbb{A}}$, $v\in V_{\mathbb{A}}$ and $k\in K$, let
\begin{equation*}
\varphi_{\mu}(vmk)=|\det m_1|_{\mathbb{A}}^{\mu_1+1}|\det m_2|_{\mathbb{A}}^{\mu_2-1}
\end{equation*}
be the $K$-invariant element of $I_P^G(\mu)$, normalized so $\varphi_{\mu}(\id)=1$. We define
\begin{equation*}
E_P(g;\mu)=E_P(g,\varphi_{\mu},\mu)
\end{equation*}
We also let $E_P$ denote the associated function on $G_{\mathbb{Z}}\backslash X_{\infty}^{+}$: for positive definite symmetric $x=g\adot \id$, $$E_P(x;\mu)=E_P(g;\mu).$$

Next we find an expression of $E_P(x;\mu)$ in terms of $x\in X_{\infty}^+$.
For  $x\in X_{\infty}^+$, we denote by $d_2(x)$ the determinant of the lower right $2\times 2$ block of $x$. Observe that $d_2(x)>0$ because $x$ is positive definite. If $g\in G_{\infty}$, by the Iwasawa decomposition we can write $g=vmk$ with $v\in V_{\infty}$, $m\in M_{\infty}$ and $k\in K_{\infty}$. From this it follows that

$$\det(g\adot e)=|\det m_1|^2|\det m_2|^2 \qquad\text{and}\qquad d_2(g.e)=|\det m_2|^2,$$
which implies

$$\det (g\adot e)^{(\mu_1+1)/2}\cdot d_2(g.e)^{-(\mu_1-\mu_2+2)/2}=|\det m_1|^{\mu_1+1}|\det m_2|^{\mu_2-1}=\varphi_{\mu}(g).$$
Using this and the natural bijection $P_{\mathbb{Z}}\backslash G_{\mathbb{Z}}\simeq P_{\mathbb{Q}}\backslash G_{\mathbb{Q}}$, we express $E_P(x,\mu)$ as a function in $G_{\mathbb{Z}}\backslash X_{\infty}^{+}$ in the following way,

\begin{equation}\label{EisenPosMat}
E_P(x,\mu)=\det x^{(\mu_1+1)/2}\sum_{\delta\in P_{\mathbb{Z}}\backslash G_{\mathbb{Z}}}d_2(\delta\adot x)^{-(\mu_1-\mu_2+2)/2}.
\end{equation}
Once again, this expression is valid for $\re(\mu_1-\mu_2)$ sufficiently large.


\section{Eisenstein series and representation numbers}
\label{sec:rep}


For $x\in X_{\mathbb{Q}}$, let $Q_x$ denote the quadratic form
associated with the matrix $x$, i.e. $Q_x(\xi)={}^t\xi x\xi$ for
$\xi\in\mathbb{R}^4$. Let $x\in X_{\mathbb{Q}}\cap X_{\infty}^{+}$
be integral, i.e. $Q_x(\xi)\in \mathbb{Z}$ for all
$\xi\in \mathbb{Z}^4$. We show that for such $x$, the Eisenstein
series $E_P(x;\mu)$ is a Dirichlet series in $\mu_1-\mu_2$. We
interpret the coefficients in terms of a type of representation
number, which counts points on the (partial) flag variety
$P_{\mathbb{Q}}\backslash G_{\mathbb{Q}}$. To define the
representation numbers we will use the Pl\"{u}cker coordinates of the
flag variety. To any $g\in G_{\mathbb{Q}}$, we associate
$v_2(g)\in \mathbb{Q}^6$, the vector of all $2\times 2$ minors in the
botton rows of $g$. For a vector $v\in\mathbb{Q}^6$, we denote by $[v]$
the associated point in the projective space
$\mathbb{P}_{\mathbb{Q}}^5$. The map
$$P_{\mathbb{Q}}g\mapsto [v_2(g)]$$
is an embedding
$$P_{\mathbb{Q}}\backslash G_{\mathbb{Q}}\hookrightarrow \mathbb{P}_{\mathbb{Q}}^5,$$
and if $(a:b:c:d:e:f)$ are the projective coordinates in $\mathbb{P}_{\mathbb{Q}}^5$, the image is the set of $(a:b:c:d:e:f)\in\mathbb{P}_{\mathbb{Q}}^5$ such that
\begin{equation}
  \label{eq:7}
 af-be+cd=0,
\end{equation}
see, e.g.,  \cite[Example 1.24]{ShafaBas}.
It will be more convenient for us to use the identification
$P_{\mathbb{Z}}\backslash G_{\mathbb{Z}}\simeq
P_{\mathbb{Q}}\backslash G_{\mathbb{Q}}$ and work with integral
coordinates. Thus the map
\begin{equation}\label{PlucOvZ1}
  \begin{split}
  P_{\mathbb{Z}}\backslash G_{\mathbb{Z}}&
\hookrightarrow \mathbb{Z}^6/\{\pm 1\}\\
  g&\mapsto [v_2(g)]
  \end{split}
\end{equation}
identifies the quotient $P_{\mathbb{Z}}\backslash G_{\mathbb{Z}}$ with 6-tuples of relatively prime integers satisfying \eqref{eq:7}.

In order to define the representation numbers of interest to us, begin by identifyting $\bigwedge^2\mathbb{Z}^4\simeq \mathbb{Z}^6$ by taking the
basis $\{e_i\wedge e_j\}_{1\leq i<j\leq 4}$ in lexicographic order. Let $Q_{\wedge^2 x}$ be the
quadratic form corresponding to the operator $\wedge^2 x$ in
$\bigwedge^2\mathbb{Z}^4$.
Explicitly the matrix of $Q_{\wedge^2 x}$ on the canonical basis of
$\mathbb{Z}^6$ is given by the matrix of principle $2\times 2$ minors
of $x$. Define the representation numbers
\begin{equation*}
r_P(x;k)=\# \{v\in P_{\mathbb{Z}}\backslash G_{\mathbb{Z}} : \ Q_{\wedge^2 x}(v)=k\}
\end{equation*}
for $k$  a positive integer.  Similarly define the genus representation numbers
\begin{equation*}
r_P(\gen(x);k)=\sum_{y\in[[x]]/\sim}\epsilon(y)^{-1}r_P(y;k).
\end{equation*}
The two associated Dirichlet series are
\begin{align*}
  Z_P(x;s)&=\sum_{k\geq 1}\dfrac{r_P(x;k)}{k^s},\\
Z_P(\gen(x);s)&=\sum_{k\geq 1}\dfrac{r_P(\gen(x);k)}{k^s}.
\end{align*}
By \cite[Lem 3.2]{CO1} we have the identity
\begin{equation*}
d_2(\delta\adot x)=Q_{\wedge^2x}(v_2(\delta)).
\end{equation*}
This together with Equation \eqref{EisenPosMat} implies the following.

\begin{prop}
  Consider $x\in X_{\mathbb{Q}}\cap X_{\infty}^{+}$ and $Q_x$
  integral. Then
  $$E_P(x;\mu)=\det x^{(\mu_1+1)/2}Z_P(x;(\mu_1-\mu_2+2)/2).$$
  If $\theta\in G_{\infty}$ is such that $\theta\adot e=x$, then
  \begin{align*}
    \int_{H_{\mathbb{Q}}^x\backslash H_\mathbb{A}^x}& E_P(h\theta;\mu)dh\\
&=\vol((H_{\mathbb{A}_f}^x\cap K_f)H_{\infty}^x)\det x^{(\mu_1+1)/2}Z_P(\gen(x);(\mu_1-\mu_2+2)/2).
  \end{align*}
\end{prop}
Therefore the problem of comparing the orthogonal period of $E_P(g;\mu)$ with Weyl group multiple Dirichlet series is reduced to the study of the representation numbers $r_P(\gen(x);k)$, which we do in the next section for $x$ equal to the 4-by-4 identity matrix $\id.$
\bigskip


\section{Explicit formula for the representation numbers}
\label{sec:formula}

Specialize now to the case of the quaternary quadratic form $Q_4=Q_{I_4}$ defined by $I_4$, that is,
\begin{equation*}
Q_{4}(x,y,z,w)=x^2+y^2+z^2+w^2.
\end{equation*}
In this case the genus $[[\id]]$ contains only one class,
cf. \cite[Chap. 9 Sec. 4 Cor. 2]{Cas78}. Following ideas from
\cite{CO2, AEW19} we compute the numbers $\rp(d):=r_P(\id;d)$ in terms
of class numbers of imaginary quadratic fields.

\subsection{Representation numbers of planes in $\Z^4$}
Using the identification $\bigwedge^2\mathbb{Z}^4\simeq \mathbb{Z}^6$ of the previous section, we
obtain
\begin{equation*}
Q_{\wedge^2\id}(v)=a^2+b^2+c^2+d^2+e^2+f^2
\end{equation*}
for $v=(a,b,c,d,e,f)\in\mathbb{Z}^6$ or its image $[v]$ in $\Z^6/\{\pm 1\}.$  Therefore $\rp(n)$ is the number of 6-tuples
$(a:b:c:d:e:f)\in\mathbb{Z}^6/\{\pm 1\}$ satisfying
\begin{itemize}
\item $af-be+cd=0$,
\item $a^2+b^2+c^2+d^2+e^2+f^2=n$, and
\item $\gcd(a,b,c,d,e,f)=1$.
\end{itemize}

Recall that a lattice $L\subset\mathbb{Z}^4$ is primitive if $(\mathbb{Q}\cdot L)\cap\mathbb{Z}^4=L$, i.e. if $L$ is generated by some vectors of a basis of $\mathbb{Z}^4$. Let $\gr$ denote the set of 2-dimensional
primitive lattices of $\mathbb{Z}^4$. We have a natural embedding
\begin{equation}
  \label{eq:8}
\begin{array}{cccc}
    \Psi: & \gr & \longrightarrow & \bigwedge\nolimits^2\mathbb{Z}^4/\{\pm 1\}, \\
     & \langle u,v\rangle & \longmapsto &  u\wedge v
  \end{array}
\end{equation}
which corresponds to the embedding  $P_{\mathbb{Z}}\backslash G_{\mathbb{Z}}$ of \eqref{PlucOvZ1}.

For $L\in \gr$, we
denote by $Q_L$ the restriction of the quadratic form
$$Q_4(x,y,z,w)=x^2+y^2+z^2+w^2$$
to $L$ and let $\disc(Q_L)$ be $-4$ times the determinant of the
matrix of $Q_L$ with respect to some basis of $L$.

\begin{lem}\label{EquaDisc}
If $L=\langle u,v\rangle\in \gr$ then
$$\disc(Q_L)=-4\cdot Q_{\wedge \id}(u\wedge v).$$
\end{lem}

\begin{proof}
If $u=(x_1,y_1,z_1,w_1)$ and $v=(x_2,y_2,y_2,w_2)$, we have
\begin{align*}
u\wedge v= & (x_1y_2-x_2y_1)e_1\wedge e_2+(x_1z_2-x_2z_1)e_1\wedge e_3+(x_1w_2-x_2w_1)e_1\wedge e_4 \\
 & +(y_1z_2-y_2z_1)e_2\wedge e_3+(y_1w_2-y_2w_1)e_2\wedge e_4+(z_1w_2-z_2w_1)e_3\wedge e_4
\end{align*}
while the Gram matrix of the binary quadratic form $Q_L$ with respect to the basis $u,v$ is
\begin{equation*}
  \begin{pmatrix}
    x_1^2+y_1^2+z_1^2+w_1^2& x_1x_2+y_1y_2+z_1z_2+w_1w_2\\  x_1x_2+y_1y_2+z_1z_2+w_1w_2 & x_2^2+y_2^2+z_2^2+w_2^2.
  \end{pmatrix}
\end{equation*}
From these two expressions the identity follows easily.
\end{proof}

Similar direct computations lead to the next result.
\begin{prop}\label{WedgePerp}
If $L\in Gr_{2,4}(\mathbb{Z})$ and $\Psi(L)=(a:b:c:d:e:f)$, then
$$\Psi(L^{\bot})=(f:-e:d:c:-b:a).$$
\end{prop}
We omit the proof.

\subsection{The Klein map}  In order to express $\rp(n)$ in terms of squares of class numbers we need another
parametrization of $\gr$ as described in \cite{AEW19}.
We denote by $\mathbf{B}(\mathbb{Q})$ the $\mathbb{Q}$-algebra of Hamilton quaternions, by $\overline{x}$ the conjugate of any element $x\in\mathbf{B}(\mathbb{Q})$ and by $Tr(x)=x+\overline{x}$ the (reduced) trace. The (reduced) norm on $\mathbf{B}(\mathbb{Q})$ is given by
\[Nr(x)=x\overline{x}=\overline{x}x=x_0^2+x_1^2+x_2^2+x_3^2,\] where
$x=x_0+x_1\mathbf{i}+x_2\mathbf{j}+x_3\mathbf{k}\in\mathbf{B}(\mathbb{Q})$.
We denote by $\mathbf{B}(\mathbb{Z})$ the subring of
$\mathbf{B}(\mathbb{Q})$ of quaternions with integral
coefficients. Let $\mathbf{B}_0(\mathbb{Q})$ denote the subset of
trace zero quaternions and $\mathbf{B}_0(\mathbb{Z})$ the trace zero
quaternions with integral coefficients.  Identify $\mathbb{Q}^4$
with $\mathbf{B}(\mathbb{Q})$ via the map
$(a,b,c,d)\mapsto a+b\mathbf{i}+c\mathbf{j}+d\mathbf{k}$.
This gives a corresponding identification of $\mathbb{Q}^3$
with $\mathbf{B}_0(\mathbb{Q})$ and of $\mathbb{Z}^3$
with $\mathbf{B}_0(\mathbb{Z})$.

Following \cite{AEW19} define
\begin{equation*}
\mathbf{K}(\mathbb{Z})=
\{(a_1,a_2) \ | \ a_1,a_2\in\mathbf{B}_0(\mathbb{Z})\backslash \{0\}
\mbox{ and }Nr(a_1)=Nr(a_2)\}/\sim
\end{equation*}
where $(a_1,a_2)\sim(a_1^{\prime},a_2^{\prime})$ if there is
$\lambda\in\{\pm 1\}$ with
$(a_1,a_2)=(\lambda a_1^{\prime},\lambda a_2^{\prime})$. We denote by $[a_1,a_2]$ the equivalence class of $(a_1,a_2)$ in $\mathbf{K}(\mathbb{Z})$. If
$L\in \gr$ with $L=\langle u,v\rangle$, put
\begin{align}\label{eq:k}
 a_1(L)&:=u\overline{v}-\frac{1}{2}Tr(u\overline{v})\\ \nonumber
a_2(L)&:=\overline{v}u-\frac{1}{2}Tr(\overline{v}u).
\end{align}
The \textit{Klein map} is defined by
\begin{equation}\label{eq:klein_map}
\begin{array}{cccc}
  \Phi: & \gr & \longrightarrow & \mathbf{K}(\mathbb{Z}) . \\
   & L & \longmapsto & [a_1(L),a_2(L)].
\end{array}
\end{equation}
We say that a pair of vectors
$(w_1,w_2)\in\mathbb{Z}^3\times\mathbb{Z}^3$ is pair-primitive if
$\frac{1}{p}w_1\notin\mathbb{Z}^3$ or
$\frac{1}{p}w_2\notin\mathbb{Z}^3$ for all odd
primes $p$ and if $\frac{1}{4}(w_1+w_2)\notin\mathbb{Z}^3$ or
$\frac{1}{4}(w_1-w_2)\notin\mathbb{Z}^3$.
The following result is proven in \cite{AEW19}.

\begin{prop}[{\cite[Prop. 2.2 and Lem. 2.4]{AEW19}}]\label{prop:K}
  The Klein map $\Phi$ is a well-defined bijection between $\gr$ and the
  set of $[a_1,a_2]\in\mathbf{K}(\mathbb{Z})$ such that $(a_1,a_2)$
  is pair-primitive and $a_1\equiv a_2\pmod 2$. Moreover, we have
  \begin{enumerate}
    \item $\disc(Q_L)=-4\cdot Nr(a_1(L))=-4\cdot Nr(a_2(L))$.
    \item $\Phi(L^{\bot})=[a_1(L),-a_2(L)]$.
  \end{enumerate}
\end{prop}

\subsection{Relation to sums of three squares}
We have $\rp(n)=\#\Rp(n)$, where $\Rp(n)$ is
the set of $L\in Gr_{2,4}(\mathbb{Z})$ such that $\disc(Q_L)=-4n.$
Therefore $\rp(n)$ is the number of pair-primitive $(a_1,a_2)$ with
$a_1\equiv a_2\pmod 2$ and $Nr(a_1)=Nr(a_2)=n$, modulo $\{1,-1\}$. We
can compute $\rp(n)$ using ideas from \cite{AEW19}, in particular,
their Proposition 2.6 and the arguments of Corollary 2.7 and Corollary
2.9. We define
$$\mathbb{D}:=\{D\in\mathbb{N} \ | \ D\not\equiv 0,7,12,15 \pmod {16}\}.$$

\begin{prop}\label{prop:r}
Let $d$ be a positive integer. We have $\rp(d)>0$ if and only if $d\in\mathbb{D}$. Let $d\in\mathbb{D}$ and write $d=d_04^ef^2$ with $d_0$ squarefree,
$f$ odd and $e\in\{0,1\}$. Then we have:
\begin{equation}\label{Formulrd}
\rp(d)=c_dr_3(d_0)^2f^2\sum_{c|f}\frac{2^{\omega(c)}}{c}\prod_{p|f}
\left(1-p^{-1}\left(\frac{-d_0}{p}\right)\right)^{e_p(f/c)},
\end{equation}
where
\begin{center}
$c_d=$
$\left\{
   \begin{array}{ll}
     1/2 & \hbox{if $d\equiv 3\pmod 4$}, \\
     1/6 & \hbox{if $d\equiv 1,2\pmod 4$}, \\
     1/3 & \hbox{if $d\equiv 0\pmod 4$},
   \end{array}
 \right.$
\end{center}

\begin{center}
$e_p(n)=$
$\left\{
   \begin{array}{ll}
     2 & \hbox{if $p| n$}, \\
     1 & \hbox{if $p\nmid n$},
   \end{array}
 \right.$
\end{center}
$\left(\tfrac{-d_0}{p}\right)$ is the Legendre symbol and $\omega(c)=\sum_{p|c}1$.
\end{prop}

\begin{proof}
If $d$ is a positive integer, we denote by $\rvprim(d)$ the number of triples $(a,b,c)\in\mathbb{Z}^3$ such that $a^2+b^2+c^2=d$ and $gcd(a,b,c)=1$.
We recall Legendre's theorem, which says that $\rvprim(d)>0$ if and only if $d\not\equiv 0,4,7 \pmod 8$. Therefore $\rp(d)=0$ if $d\equiv 7\pmod 8$.

Let $d\equiv 1,2$ or $3\pmod 4$.
  As in the proof of Corollary 2.7 in \cite{AEW19}, we see that $\rp(d)$ is
  half the number of pair-primitive tuples $(v,v^{\prime})$ such that
  $Nr(v)=Nr(v^{\prime})=d$ and $v\equiv v^{\prime}\pmod 2$. The pair-primitive tuples $(v,v^{\prime})$ are precisely of the form
  $(cw,c^{\prime}w^{\prime})$ with $Nr(w)=\frac{d}{c^2}$,
  $Nr(w^{\prime})=\frac{d}{c^{\prime 2}}$, $w$ and $w^{\prime}$
  primitive vectors, $gcd(c,c^{\prime})=1$ and
  $w\equiv w^{\prime}\pmod 2$. If $d\equiv 3 \pmod 4$ the last congruence condition is automatically satisfied. So writing $d=d_0f^2$ with
  $d_0$ squarefree, we obtain
\begin{equation}\label{Formulrd1}
  \rp(d)=\frac 12\sum_{\substack{c,c^{\prime}|f\\ gcd(c,c^{\prime})=1}}
\rvprim\left(\frac{d}{c^2}\right)\rvprim
\left(\frac{d}{c^{\prime 2}}\right).
\end{equation}

If $d\equiv 1,2\pmod 4$ the only difference is that once $w$ is fixed, the congruence condition $w\equiv w^{\prime}\pmod 2$ cuts down the possibilities for $w^\prime$ by a third.  This is because $w^{\prime}$ must have precisely one odd coordinate (if $d\equiv 1 \pmod 4$)  or one even coordinate (if $d\equiv 2 \pmod 4$) in the same position as the corresponding coordinate in $w.$   Thus
  \begin{equation}\label{Formulrd2}
    \rp(d)=\frac{1}{2}\cdot\frac 13\sum_{\substack{c,c^{\prime}|f\\ gcd(c,c^{\prime})=1}}
\rvprim\left(\frac{d}{c^2}\right)\rvprim\left(\frac{d}{c^{\prime 2}}
\right).
  \end{equation}

Let $d\equiv 0 \pmod 4$.
  In this case the pair-primitives tuples $(w,w^{\prime})$ with $Nr(w)=Nr(w^{\prime})=d$ and $w\equiv w^{\prime}\pmod 2$ are of the form $(2v,2v^{\prime})$ with $(v,v^{\prime})$ pair-primitive, $Nr(v)=Nr(v^{\prime})=\frac{d}{4}$ and $v\not\equiv v^{\prime}\pmod 2$.  In particular we have $\rp(d)=0$ if $\frac{d}{4}\equiv 0,3\pmod 4$. Next we suppose that $\tfrac{d}{4}\equiv 1,2\pmod 4$. The pair-primitive tuples
  $(v,v^{\prime})$ are of the form $(cw,c^{\prime}w^{\prime})$ with
  $Nr(w)=\frac{d}{4c^2}$, $Nr(w^{\prime})=\frac{d}{4c^{\prime 2}}$,
  $w$ and $w^{\prime}$ primitive vectors, $gcd(c,c^{\prime})=1$ and $w\not\equiv w^{\prime}\pmod 2$. This reduces the choices of
  $w^{\prime}$ to two thirds of the options. So writing
  $\frac{d}{4}=d_0f^2$ with $d_0$ squarefree and $f$ odd, we obtain
  \begin{equation}\label{Formulrd3}
\rp(d)=\frac 12\cdot\frac{2}{3}\sum_{\substack{c,c^{\prime}|f\\ gcd(c,c^{\prime})=1}}
\rvprim\left(\frac{d}{4c^2}\right)
\rvprim\left(\frac{d}{4c^{\prime 2}}\right).
 \end{equation}

Let $n$ be a positive integer and write $n=n_0m^2$ with $n_0$ the squarefree part of $n$. We suppose that $m$ is odd.
We have the following formula for $\rvprim(n)$ (cf. \cite{ch07}):
\begin{equation*}
\rvprim(n)=r_3(n_0)m
\prod_{p|m}\left(1-p^{-1}\left(\frac{-n_0}{p}\right)\right)
\end{equation*}
Thus
$$\rvprim\left(\frac{d}{c^2}\right)
\rvprim\left(\frac{d}{c^{\prime 2}}\right)
=r_3\left(d_0\right)^2\frac{f^2}{cc^{\prime}}
\prod_{p|f}\left(1-p^{-1}\left(\frac{-d_0}{p}\right)\right)^{e_p},$$
where
\begin{center}
$e_p=e_p\left(\tfrac{f}{cc^{\prime}}\right)=$
$\left\{
   \begin{array}{ll}
     2 & \hbox{if $p| \frac{f}{cc^{\prime}}$}, \\
     1 & \hbox{if $p\nmid \frac{f}{cc^{\prime}}$.}
   \end{array}
 \right.$
\end{center}
Using this expression in \eqref{Formulrd1}, \eqref{Formulrd2} and \eqref{Formulrd3} we obtain \eqref{Formulrd}.
\end{proof}

If $d_0> 3$, we can also write Equation \eqref{Formulrd} in terms of class
numbers using the formulas
\begin{equation}\label{r3ClassNum}
r_3(d_0)=
\left\{
   \begin{array}{ll}
     24h_K & \hbox{when $d_0\equiv 3 \pmod 8$}, \\
     12h_K & \hbox{when $d_0\equiv 1,2 \pmod 4$},
   \end{array}
 \right.
\end{equation}
where $K=\mathbb{Q}(\sqrt{-d_0})$ and $h_K$ is the class number of $K$ (see \cite[Prop. 2.3]{EMV}).

\subsection{The Dirichlet series $\rs(w)$.}

Proposition \ref{prop:r} also allows us to obtain expressions for the
Dirichlet series
\begin{equation}
  \label{eq:rseries}
\rs(w):=Z_P(\id;w)=\sum_{d=1}^{\infty} \frac{\rp(d)}{d^{w}}.
\end{equation}
As these formulas depend on $d$ mod 4 we split our sum into
distinct congruence classes mod 4.
For simplicity, we focus on the case $d\equiv
3\pmod 4$.  The other case are handled similarly. Let
\begin{equation}
  \label{eq:rseries2}
\rs^{(3)}(w)=\sum_{d\equiv 3\!\!\!\!\!\pmod {\! 4}}
\frac{\rp(d)}{d^{w}}
=\sum_{\substack{d_0\equiv 3\!\!\!\!\!\pmod {\! 4}\\
\square\mbox{\scriptsize -free}}}\sum_{f\geq 1, \mbox{\scriptsize odd}}
\frac{\rp(d_0f^2)}{(d_0f^2)^{w}}.
\end{equation}
In fact, since $\rp(d)=0$ for $d\equiv 7\pmod{8}$, the above sum is only
over $d\equiv 3\pmod{8}$. 
\begin{thm}\label{thm:rs3}
  For $\re(w)$ sufficiently large,
  \begin{equation}
    \label{eq:rs3thm}
\rs^{(3)}(w)=\sum_{\substack{d_0\equiv 3\!\!\!\!\!\pmod {\! 4}\\
\square\mbox{\scriptsize -free}}}
\frac{r_3(d_0)^2P_{d_0}(w)}{d_0^{w}}
  \end{equation}
where $P_{d_0}$ is given by the Euler product
$$P_{d_0}(w)=\prod_p P\left(p^{-w},
{\scriptstyle {\left(\frac {-d_0}{p}\right)}}\right),$$
and
\begin{equation}
  \label{def:Pp}
  P(y,\epsilon)=
\frac{1+(\epsilon^2-2\epsilon+p-2\epsilon p)y^2+\epsilon^2py^4}
{(1-py^2)(1-p^2y^2)}.
\end{equation}
\end{thm}

\begin{proof}
  By (\ref{eq:rseries}) and Proposition \ref{prop:r}, we have
$$\rs^{(3)}(w)=\sum_{\substack{d_0\equiv 3\!\!\!\!\!\pmod {\! 4}\\
\square\mbox{\scriptsize -free}}}\frac{r_3(d_0)^2}{d_0^{w}}
\left[\sum_{f\geq 1, \mbox{\scriptsize odd}}
f^{2-2w}
\sum_{c|f}\dfrac{2^{\omega(c)}}{c}
\prod_{p|f}\left(1-p^{-1}\left(\dfrac{-d_0}{p}\right)\right)^{e_p(f/c)}
\right].$$
The inner sum over $f$ is an Euler product, and setting
$y=p^{-w}, A=1-p^{-1}\left(\tfrac{-d_0}{p}\right)$, its $p$-part is
\begin{equation}
  \label{eq:ppartP}
  P(y,\epsilon)=1+p^2y^2A\left(A+\frac{2}{p}\right) +\sum_{k=2}^\infty A p^{2k}y^{2k}
  \left[A+\frac{2A}{p}\frac {1-p^{1-k}}{1-p^{-1}} +\frac{2}{p^k}\right]
\end{equation}
for $\epsilon=\left(\tfrac{-d_0}{p}\right)$. Summing the geometric
series above and combining, we arrive at (\ref{def:Pp}).
\end{proof}


\section{Weyl group multiple Dirichlet series}
\label{sec:wmds}

In this section, we show that the Dirichlet series constructed from the
coefficients counting planes in $\Z^4$ coincides with a specialization
of a multiple Dirichlet series arising in the Fourier expansion of the
minimal parabolic Eisenstein series on a metaplectic double cover of
$\GL(4).$  In a more general context Brubaker, Bump and Friedberg \cite{bbf-xtal}
have expressed the Fourier coefficients of the Eisenstein series on the
$n$-fold cover of $\GL(r)$ in terms of crystal bases.  We will instead
use formulas of Chinta and Gunnells which Brubaker, Bump, Friedberg
and Hoffstein \cite{bbfh} have shown to be equal to the ones in
\cite{bbf-xtal}. Actually \cite{bbf-xtal} works over a number field
containing a $4^{th}$ root of unity; the formulas over $\Q$ require a
modification at the prime $2$, which fortunately plays no role in the
present work.  We refer the reader to Karasiewicz \cite{edk} for the
analogous formulas on the double cover of $\GL(3)$ over $\Q$.

We now define the multiple Dirichlet series to which we must compare
$\rs(w)$.  This is the $A_3$ quadratic Weyl group multiple Dirichlet
series and arises in the Fourier expansion of the Borel Eisenstein
series on a metaplectic double cover of $\GL(4).$ There are various
ways to define this series, but we follow the presentation of
Chinta-Gunnells \cite{cg-inv}.

For a quadratic character $\chi:(\Z/d\Z)^\times\to \{1,-1\}$ define the Dirichlet series
\begin{align*}
  L(s,\chi)&=\prod_p\left(1-\frac{\chi(p)}{p^s}\right)^{-1}
    = \sum_{n=1}^\infty \frac{\chi(n)}{n^s}, \mbox{ and}\\
  L_2(s,\chi)&=\prod_{p\neq 2}\left(1-\frac{\chi(p)}{p^s}\right)^{-1}    =
\sum_{n \text{ odd}} \frac{\chi(n)}{n^s}.\\
\end{align*}
Let $\psi_1, \psi_2, \psi_3$ be three
primitive, quadratic Dirichlet characters unramified away from 2. Thus
each of the $\psi_i$ is either trivial or one of
$\chi_{-4}, \chi_8, \chi_{-8}.$ Define
\begin{equation}
  \label{def:A3Z}
Z_{A_3}(s_1,s_2,w;\psi_1,\psi_2,\psi_3)=\sum_{\substack{d,n_1,n_2>0\\ odd}}
\frac{\chi_{d'}(\hat n_1)\chi_{d'}(\hat n_2)}{n_1^{s_1}n_2^{s_2}
d^{w}}a(n_1,n_2,d)\psi_1(n_1)\psi_2(n_2)\psi_3(d),
\end{equation}
where
\begin{itemize}
\item $d'=(-1)^{(d-1)/2} d$ and $\chi_{d'}$ is the Kronecker
  symbol associated to the squarefree part of $d'$
\item $\hat n$ is the part of $n$ relatively prime to the
  squarefree part of $d$
\item the coefficients $a(n_1,n_2,d)$ are weakly multiplicative in all
entries and are defined on prime powers by
\begin{align}
  \label{def:A3a}\nonumber
H(x_1,x_2,y)&=\sum_{k,l,m} a(p^k,p^l,p^m) x_1^kx_2^ly^m\\
&=
\frac{1-x_1y-x_2y+x_1x_2y+px_1x_2y^2-px_1x_2^2y^2-px_1^2x_2y^2-px_1^2x_2^2y^3}
{(1-x_1)(1-x_2)(1-y)(1-px_1^2y^2)(1-px_2^2y^2)(1-p^2x_1^2x_2^2y^2)}.
\end{align}
\end{itemize}
As shown in \cite{cg-inv}, we can write this as
\begin{equation}\nonumber
  \label{eq:Zqmds}
  \begin{split}
Z_{A_3}(s_1,s_2,w;&\psi_1,\psi_2,\psi_3)=\\
&\sum_{\substack{d_0>0\\ odd,\square\mbox{-}free }}
\frac{L_2(s_1,\chi_{d_0'}\psi_1)
L_2(s_2,\chi_{d_0'}\psi_2)\psi_3(d_0)}{d_0^w}
Q_{d_0}(s_1,s_2,w;\psi_1,\psi_2),
  \end{split}
\end{equation}
say, where $Q_{d_0}$ is the Euler product
\begin{equation}\label{def:Qd}
Q_{d_0}(s_1,s_2,w;\psi_1,\psi_2)=
\prod_{p\ odd}Q_{d_0,p}(\epsilon_{1,p}p^{-s_1},
\epsilon_{2,p}p^{-s_2},p^{-w})
\end{equation}
with $\epsilon_{1,p}=\chi_{d_0'}(\hat p)\psi_1(p),
\epsilon_{2,p}=\chi_{d_0'}(\hat p)\psi_2(p)$ and
\begin{equation}
  \label{def:Qp}
Q_{d_0,p}(x_1,x_2,y)=
\begin{cases}
  \frac {H(x_1,x_2,y)+H(x_1,x_2,-y)}{2}(1-x_1)(1-x_2)
& \mbox{if }p\nmid d_0,\\
  \frac {H(x_1,x_2,y)-H(x_1,x_2,-y)}{2}
& \mbox{if }p| d_0.\\
\end{cases}
\end{equation}

To go further it is convenient to divide the sum over $d$ into congruence
classes mod $8$.  As in the computation of $\rs^{(3)}(s)$ above, we will
concentrate on the case $d\equiv 3\pmod{8}.$  Define
\begin{align*}
  Z_{A_3}^{(3)}(s_1,s_2,w)&= \tfrac 14\left[Z_{A_3}(s_1,s_2,w;1,1,1)
-Z_{A_3}(s_1,s_2,w;1,1,\chi_{-4})\right.\\
&\qquad \qquad\left.
{}- Z_{A_3}(s_1,s_2,w;1,1,\chi_8)+Z_{A_3}(s_1,s_2,w;1,1,\chi_{-8})\right] \\
&=\sum_{\substack{0<d_0\equiv 3\!\!\!\!\!\pmod {\! 8}\\
\square\mbox{\scriptsize -free}}}
\frac{L_2(s_1,\chi_{-d_0})
L_2(s_2,\chi_{-d_0})}{d_0^w}
Q_{d_0}(s_1,s_2,w;1,1).
\end{align*}
If $s_1=s_2=1$, then we have
\begin{align}
  \label{eq:Z_s=1}
Z_{A_3}^{(3)}(1,1,w)&= \sum_{\substack{0<d_0\equiv 3\!\!\!\!\!\pmod {\! 8}\\
\square\mbox{\scriptsize -free}}}
\frac{L_2(1,\chi_{-d_0})^2}{d_0^w}Q_{d_0}(1,1,w;1,1)\\ \nonumber
&=\frac{9}{4}\sum_{\substack{0<d_0\equiv 3\!\!\!\!\!\pmod {\! 8}\\ \nonumber
\square\mbox{\scriptsize -free}}}
\frac{L(1,\chi_{-d_0})^2}{d_0^w}Q_{d_0}(1,1,w;1,1)\\ \nonumber
&=\frac{9}{4}\cdot\left(\frac{\pi^2}{576}\right)
\sum_{\substack{0<d_0\equiv 3\!\!\!\!\!\pmod {\! 8}\\
\square\mbox{\scriptsize -free}}}
\frac{r_3(d_0)^2}{d_0^{w+1}}Q_{d_0}(1,1,w;1,1)
\end{align}
where we have used that for squarefree $d_0\equiv 3\pmod 8$. We have
\begin{equation}
  \label{eq:L2}\nonumber
1-\frac{\chi_{-d_0}(2)}{2}=\frac 32,
\end{equation}
and by \cite[Chap. 5, 1.1 Thm. 2]{BoreShafa} and \eqref{r3ClassNum}, it follows that
\begin{equation}
  \label{eq:r3}\nonumber
L(1,\chi_{-d_0})=\frac {\pi r_3(d_0)}{24\sqrt {d_0}}.
\end{equation}
Let us write $Q_{d_0}(w)$ for $Q_{d_0}(1,1,w;1,1).$

\begin{thm}
  \label{thm:PQ}
We have
$$\frac{\pi^2}{256}\zeta_2(2w)\zeta_2(2w-1)\rs^{(3)}(w)=Z_{A_3}^{(3)}(1,1,w-1).
$$
\end{thm}

\begin{proof}
Comparing  (\ref{eq:rs3thm}) with the last line of (\ref{eq:Z_s=1})
we see that we need to prove
\begin{equation}\label{eq:PQid}
\zeta_2(2w)\zeta_2(2w-1)P_{d_0}(w)=Q_{d_0}(w-1)
\end{equation}
As both sides are Euler products it suffices to show that the
$p$-parts match, for all odd primes $p$. Let $\epsilon=\chi_{d_0}(p).$
From (\ref{def:Pp}) of Theorem \ref{thm:rs3} the $p$-part of the
lefthand side of (\ref{eq:PQid}) is
\begin{equation}
\label{eq:Pepsilon}
\frac{P(y,\epsilon)}{(1-y^2)(1-py^2)}=
\begin{cases}
  \dfrac{1}{(1-py^2)(1-p^2y^2)}& \mbox{if } \epsilon=1,\\
\dfrac{1+3(p+1)y^2+py^4}{(1-y^2)(1-py^2)^2(1-p^2y^2)}& \mbox{if } \epsilon=-1,\\
\dfrac{1+py^2}{(1-y^2)(1-py^2)^2(1-p^2y^2)}& \mbox{if } \epsilon=0.\\
\end{cases}
\end{equation}
On the other hand, the $p$-part of the righthand side is given in
(\ref{def:Qd}), (\ref{def:Qp}) to be
\begin{equation}
  \label{eq:Qdp}
Q_{d_0,p}(\tfrac 1p, \tfrac 1p, py)=
\begin{cases}
    \frac {H(\tfrac 1p,\tfrac 1p, py)+H(\tfrac 1p,\tfrac 1p, -py)}{2}
(1-\tfrac 1p)^2
& \mbox{if }\epsilon=1,\\
    \frac {H(-\tfrac 1p,-\tfrac 1p, py)+H(-\tfrac 1p,-\tfrac 1p, -py)}{2}
(1+\tfrac 1p)^2& \mbox{if }\epsilon=1,\\
  \frac {H(\tfrac 1p,\tfrac 1p, py)-H(\tfrac 1p,\tfrac 1p, -py)}{2}
& \mbox{if }\epsilon=0.\\
\end{cases}
\end{equation}
Using the definition of $H$ in (\ref{def:A3a}) we readily verify that
$p$-parts of (\ref{eq:Pepsilon}) and (\ref{eq:Qdp}) match up in each
of the 3 cases.
\end{proof}


\section{Relations between the four CM-points}
\label{sec:4cm}

For a given lattice $L$ in $\gr$, Aka-Einsiedler-Wieser \cite{AEW19} define four associated $CM$-points $$z_1(L),\ldots, z_4(L)\in\X:=\PGL(\Z)\backslash \PGL(\R)/\PO$$ and study the joint distribution of the sets
\begin{equation*}
\{(L,z_1(L), z_2(L), z_3(L), z_4(L)):L\in\Rp(n)\}\subset\gr\times\X^4
\end{equation*}
for large $n$.  In this section we recall the definitions of the four CM-points
and describe arithmetic relations between them.  We will phrase our results in terms of the classes of binary quadratic forms corresponding to the CM-points.  For clarity of exposition, we will present the main results of this section only for $n$ squarefree and congruent to 1 mod 4.  The case of $n\equiv 3 \pmod 4$ is similar but complicated by the appearance of imprimitive binary quadratic forms corresponding to the geometric $CM$-points defined below.

In what follows, if $q$ is a quadratic form defined on a two dimensional sublattice $M=\langle u,v\rangle$  of $\Z^d$, we let $[q]$ denote the $\GL_2(\Z)$ equivalence class of the binary quadratic form defined by
\begin{equation}
  \label{eq:q}
(x,y)\mapsto q(xu+yv).
\end{equation}
Of course the quadratic form above depends on the choice of basis, but its    $\GL_2(\Z)$ equivalence class is well defined.
The CM-points $z_1(L)$ and $z_2(L)$ are defined to be the points in $\X$ corresponding to the classes of the binary quadratic forms $[Q_L]$ and $[Q_{L^\bot}]$, respectively.  These are called the \emph{geometric CM-points} in \cite{AEW19}.  To define $z_3(L)$ and $z_4(L)$---termed the \emph{accidental CM-points} in \cite{AEW19}---we use the identification of $\gr$ with the subset of $\mathbf{K}(\mathbb{Z})$ in Proposition \ref{prop:K} provided by the Klein map $\Phi$.  Recall that
\[\Phi(L)=[a_1(L), a_2(L)]
\]
where the $a_1(L), a_2(L)$ defined in  (\ref{eq:klein_map}) are traceless integral quaternions, which we identify with $\Z^3$ in the natural way: $x\mathbf{i}+y\mathbf{j}+z\mathbf{k}\mapsto (x,y,z)$.  For $i=1,2$ define the two-dimensional sublattices $M_i=a_i(L)^\bot\subset \Z^3.$
We define $z_3(L)$ to be the point corresponding to the class of the binary quadratic form $Q|_{M_1}$ obtained by restricting $Q$ to $M_1$.  Similarly  $z_4(L)$ is the point corresponding to the form $[Q|_{M_2}]$.  For $M$ a sublattice of $\Z^3$ we use the notation $Q_{M}=Q|_{M}$.  Here we think of $M\subset \Z^3$ as a sublattice in $\Z^4$ by embedding $(a,b,c)\mapsto (0,a,b,c).$

We take this opportunity to recall Gauss's result on which binary quadratic forms can arise as a restriction of the ternary quadratic form $Q_3(x,y,z)=x^2+y^2+z^2$ to a two dimensional sublattice of $\Z^3.$  Let $\mathcal G$ be the map from nonzero vectors in $\Z^3$ to $\GL_2(\Z)$ equivalence classes of integral binary quadratic forms defined by
\begin{equation}
  \label{eq:6}
v\mapsto [Q_3|_{v^\bot}].
\end{equation}
It is easy to see that if $Q_3(v)=n$ then $\disc\mathcal G(v)=-4n$.  In fact, Gauss \cite{Gauss} proves the following sharper result about the image of $\mathcal G.$  Let $\mathcal G_n$ denote the image of $\mathcal G$ restricted to $\Rv(n)$:
\begin{equation}
  \label{eq:9}
\mathcal G_n=\{\mathcal G(v): v\in\Rv(n)\}.
\end{equation}
Then

\begin{thm}\label{GaussThMain}
Let $n\equiv 1,2 \pmod 4$ be a positive squarefree integer and $v\in\Rv(n).$
The image $\mathcal G_n$  consists of a single genus class of binary quadratic forms of discriminant $-4n$.
\end{thm}

Returning to the case of planes in $\Z^4$, let $L\in\Rp(n)$. It follows from Lemma \ref{EquaDisc} and Proposition \ref{WedgePerp} that $\disc(Q_L)=\disc(Q_{L^{\bot}})=-4n$. Consider the pairs $(L,Q_L)$ and $(L^{\bot},Q_{L^{\bot}})$. As $(L^{\bot},Q_{L^{\bot}})$ is determined by $(L,Q_L)$, we want to find an intrinsic relation between the two pairs and to also describe the relation with $(M_i, Q_{M_i})$, for $i=1,2$. We prove below that $(L,Q_L)$ determines the genus class of the three other quadratic forms.  For this we need the concept of the Legendre composition of binary quadratic forms.  A convenient reference for the definitions and results we need on Legendre composition is the paper of Towber \cite{towber}.

\begin{defn}
A \textit{quadratic lattice} is a pair $(L,q)$ where $L$ is a lattice and $q$ is a integer valued quadratic form on $L$. We say that two quadratic lattices are isomorphic if there is a linear isomorphism between them which preserves the quadratic forms. If $\rk L=2$ we call $(L,q)$ a \textit{binary quadratic lattice}.
\end{defn}

\begin{defn}[{\cite[Defn. 2.1]{towber}}]
We say that the binary quadratic lattice $(M,q_M)$ is a \textit{Legendre composition} of the binary quadratic lattices $(L,q_L)$ and $(L^{\prime},q_{L^{\prime}})$ if there is a linear and surjective homomorphism $\mu:L\otimes L^{\prime}\rightarrow M$ such that
$$q_L(u)q_{L^{\prime}}(v)=q_M(\mu(u\otimes v)).$$
\end{defn}

Now, for $L, L^\bot\in\Rp(n)$ consider the maps
$$\begin{array}{cccc}
    \mu_1: & L\otimes L^{\bot} & \longrightarrow & \mathbf{B}_0(\mathbb{Z}) \\
     & v\otimes w & \longmapsto & v\overline{w}
  \end{array}$$
and
$$\begin{array}{cccc}
    \mu_2: & L\otimes L^{\bot} & \longrightarrow & \mathbf{B}_0(\mathbb{Z}) \\
     & v\otimes w & \longmapsto & \overline{v}w.
  \end{array}$$

The next result shows that the image of $\mu_i$ is the two dimensional lattice $M_i$ for $i=1,2.$

\begin{thm}\label{CompOrtLatt}
Let $L\in\Rp(n)$ with $n\in\mathbb{D}$ squarefree and congruent to 1  mod 4.
Let $M_1,M_2\subset\Z^3$ be the quadratic spaces associated to the two accidental CM-points as defined above.
For $i=1,2$, the 2-dimensional lattice $M_i$ is the image of $\mu_i$ and $(M_i,Q_{M_i})$ is a Legendre composition of $(L,Q_L)$ and $(L^{\bot},Q_{L^{\bot}})$.
\end{thm}
\begin{proof}
Write $L=\langle v_1,v_2\rangle$ and denote by $M_1^{\prime}$ the image of $\mu_1$.
Observe that the map $w\mapsto\mu_1(v_2\otimes w)$ from $L^{\bot}$ to $\mathbf{B}_0(\mathbb{Z})$ is injective, therefore $\rk M_1^{\prime}$ is equal to $2$ or $3$. On the other hand, if $w\in L^{\bot}$, then
\begin{align*}
\langle a_1(L),v_2\overline{w}\rangle & =-\frac{1}{2}Tr(v_1\overline{v_2}v_2\overline{w}-\frac{1}{2}Tr(v_1\overline{v_2})v_2\overline{w}) \\
 & =-\frac{1}{2}Nr(v_2)Tr(v_1\overline{w})+\frac{1}{4}Tr(v_1\overline{v_2})Tr(v_2\overline{w})=0
\end{align*}
Using that $-a_1(L)=v_2\overline{v_1}-\tfrac{1}{2}Tr(v_1\overline{v_2})$ we obtain $\langle a_1(L),v_1\overline{w}\rangle =0$. Therefore $M_1^{\prime}\subset M_1$ and $\rk M_1^{\prime}=2$. If $v\in L$ and $w\in L^{\bot}$, then
$$Q_L(v)Q_{L^{\bot}}(w)=Nr(v\overline{w})=Nr(\mu_1(v\otimes w)).$$
Therefore $M_1^{\prime}$ is a Legendre composition of $(L,Q_L)$ and $(L^{\bot},Q_{L^{\bot}})$.
By the Second Conclusion in art. 235 of \cite{Gauss}, $\disc(Q_{M_1^{\prime}})$ divides $-4n=\disc(Q_L)=\disc(Q_{L^{\bot}})$. On the other hand, by Theorem \ref{GaussThMain}, $\disc(Q_{M_1})=-4n$. This fact together with $M_1^{\prime}\subset M_1$ implies that $M_1^{\prime}=M_1$. Analogously we prove that the image of $\mu_2$ is $M_2$.
\end{proof}

\begin{thm}\label{Repres}
If $q(x,y)=ax^2+2bxy+cy^2$ is a quadratic form with $a,b,c\in\mathbb{Z}$, $a,c>0$, $ac-b^2=d>0$, $d$ squarefree and $d\not\equiv 7\pmod 8$, then $q$ is represented by $Q_4$.
\end{thm}

\begin{proof}
This follows from formula (4) in \cite{Taussky}.
\end{proof}

If $(L,q_L)$ and $(M,q_M)$ are quadratic lattices, we denote by $q_L\oplus q_M$ the quadratic form in $L\oplus M$ defined by $(q_L\oplus q_M)(v+w):=q_L(v)+q_M(w)$ for $v\in L$ and $w\in M$.

\begin{lem}\label{LemRepres}
Let $(L,q_L)$, $(L^{\prime},q_{L^{\prime}})$, $(M,q_M)$ and $(M^{\prime},q_{M^{\prime}})$ be positive definite binary quadratic spaces such that
\begin{itemize}
\item $(L^{\prime},q_{L^{\prime}})$ is in the same genus class of $(L,q_L)$ and
\item  $(M^{\prime},q_{M^{\prime}})$ is in the same genus class of $(M,q_M)$.
\end{itemize}
If $(L\oplus M, q_L\oplus q_M)$ is represented by $(\mathbb{Z}^4,Q_4)$, then so is $(L^{\prime}\oplus M^{\prime}, q_{L^{\prime}}\oplus q_{M^{\prime}})$.
\end{lem}

\begin{proof}
Since $Q_4$ is the only form in its genus class, the result follows by the Hasse-Minkowski theorem.
\end{proof}

We continue to let $n\in\mathbb{Z}_{>0}$ be squarefree and equivalent to 1 mod 4 and $\cl(-4n)$ be  the abelian group of proper classes of positive primitive binary quadratic forms of discriminant $-4n$, with the group law given by Gaussian composition. In the following we will use the concepts of Gaussian composition and Legendre composition of binary quadratic forms as explained in \cite{towber} (cf.
\cite[Defn. 2.1, Defn. 2.2 and pp. 45-46]{towber}). Observe in particular that a Gaussian composition of two forms is a Legendre composition.  In the following lemma we discuss conversely the relation between Legendre composition and Gaussian composition.  If $q(x,y)=ax^2+bxy+cy^2$ is a binary quadratic form, then we define $q^{op}(x,y)=ax^2-bxy+cy^2$.

\begin{lem}\label{LegenGauss}
If $q_1$, and $q_2$ are primitive binary quadratic forms of the same discriminant and $q_3$ is a positive binary quadratic form which is a Legendre composition of $q_1$ and $q_2$, then one of the following happens:
\begin{enumerate}
  \item $q_3$ is the Gaussian composition of $q_1$ and $q_2$,
  \item $q_3^{op}$ is the Gaussian composition of $q_1$ and $q_2$,
  \item $q_3$ is the Gaussian composition of $q_1^{op}$ and $q_2$ or
  \item $q_3$ is the Gaussian composition of $q_1$ and $q_2^{op}$.
\end{enumerate}
\end{lem}

\begin{proof}
Follows from \cite[Theorem 2.2]{towber}.
\end{proof}

Observe that the quadratic forms $q$ and $q^{op}$ are in the same genus class,
indeed the product of $q$ and $q^{op}$ is the identity element in the class group $\cl(-4n).$  In particular the genus class of a binary quadratic lattice is well defined.  It follows from \cite[Chap. 14, Corollary to Theorem 3.1]{Cas78} that the genera of $\cl(-4n)$ are precisely the cosets of $\cl(-4n)$ modulo $\cl(-4n)^2$.
Therefore the Gaussian composition in $\cl(-4n)$ descends to a composition of genus classes and we obtain the following.

\begin{cor}\label{LegenGaussCor}
If $q_1,q_2,q_3$ are as in Lemma \ref{LegenGauss}, then $q_3$ is in the  genus class of the Gaussian  composition of $q_1$ and $q_2$.
\end{cor}

Finally we can prove the main theorem of this section which describes the possibilities for the pairs of binary quadratic lattices $(L,Q_L)$ and $(L^{\bot},Q_{L^{\bot}})$ for $L\in\mathcal{R}_{24}(d)$.

\begin{thm}\label{WhenExPairLat}
Let $n\in\mathbb{D}$ be a squarefree integer, $n\equiv 1\pmod 4$. Let $q_1$ and $q_2$ be two positive binary quadratic forms of discriminant $-4n$. Then there exists $L\in\mathcal{R}_{24}(n)$ with $[Q_L]=[q_1]$ and $[Q_{L^\bot}]=[q_2]$ if and only if the Legendre compositions of $q_1$ and $q_2$ belong to the genus class $\mathcal{G}_n$ defined in Theorem \ref{GaussThMain}.
\end{thm}

\begin{proof}
If $L\in\mathcal{R}_{24}(n)$, then by Theorem \ref{CompOrtLatt} and Corollary \ref{LegenGaussCor}  we conclude that the Legendre compositions of $(L,Q_L)$ and $(L^{\bot},Q_{L^{\bot}})$ are in the genus class $\mathcal{G}_n$.

Conversely, consider $q_1$ and $q_2$ two positive binary quadratic forms of discriminant $-4n$ with a Legendre composition in the genus class $\mathcal{G}_n$. By Theorem \ref{Repres} there exists a 2-dimensional primitive sublattice $N$ of $\mathbb{Z}^4$ with $[q_1]=[Q_N].$ 
As the Legendre compositions of $q_1$ and $q_2$ belong to the genus class $\mathcal{G}_n$, we obtain that $Q_{N^{\bot}}$ and $q_2$ are in the same genus class. By Lemma \ref{LemRepres} we conclude that there exists $L$ a $2$-dimensional primitive sublattice of $\mathbb{Z}^4$ such that $[Q_L]=[q_1]$ and $[Q_{L^{\bot}}]=[q_2].$  This completes the proof.
\end{proof}

\section{Further questions}
\label{sec:further}

In this paper we have studied the number of two dimensional sublattices of $\Z^4$ of a fixed discriminant and shown how these numbers arise in both nonmetaplectic and metaplectic Eisenstein on $\GL_4.$  The conceptual background linking the two Eisenstein series on different groups is provided by a conjecture of Jacquet motivated by the formalism of the relative trace formula.  From another perspective, Aka, Einsiedler and Wieser naturally associate to a plane in $\Q^4$ four CM points and prove various equidistribution results for the plane together with the CM points.  We list several open problems and potential directions for generalizations suggested by these two perspectives.

\begin{itemize}
\item  Let $n$ be squarefree and $\cl(D)$ the class group of the ring of integers of $\Q(\sqrt {-n}).$  (Hence $D=-n$ or $-4n$.)   Aka, Einsiedler and Wieser note that the Klein map provides a finite quotient of $\Rp(n)$ with the structure of $\cl(D)^2$-torsor, \cite{AEW19}[Section 8].  Given Bhargava's \cite{bhargava} identification of $\cl(D)^2$ with equivalence classes of $2\times 2\times 2$ integer cubes, it would be interesting to directly construct an action of cubes of hyperdeterminant $D$ on  $\Rp(D)$.
\item Another related setting where we expect to see a natural action of a class group is that of mutually orthogonal triples of planes in $\Z^6$.  In this case the multiple Dirichlet series which arise as the Fourier-Whittaker coefficients of the metaplectic double cover Eisenstein series on $\GL_6$ were conputed in \cite{chinta-biquad} and shown to involve class numbers of biquadratic extensions of $\Q.$
\item Let $P$ be the Siegel parabolic of the symplectic group $\Sp_4(\R).$  That is, $P$ is a maximal parabolic subgroup of  $\Sp_4(\R)$ stabilizing a two dimensional Lagrangian subspace.  Thus $P\backslash\Sp_4(\R)$ parametrizes Lagrangian planes.  Unpublished computations of Chinta, Hundley and Offen again show a correspondence between a unitary period of the associated Eisenstein series and a metaplectic Eisenstein series on the double cover of $\GL_4$.  The computation of this period should be amenable to the methods of this paper.  It would also be interesting to pursue analogues of the results of Aka, Einsiedler and Wieser and study the joint distribution of Lagrangian planes in $\Q^4$ and their associated four-tuples of CM points.
\item Our proof of Proposition \ref{prop:r} relating $\rp(D)$ to $\rv(D)^2$ is a straightforward extension of the ideas in Corollary 2.9 of \cite{AEW19} which shows that
\begin{equation}
  \label{eq:10}
\rp(D)=D^{1+o(1)} \mbox{ for $D\in\mathbb{D}$}.
\end{equation}
The authors remark that they are not aware of counting results like (\ref{eq:10}) for rational subspsaces of dimension $k$ and discriminant $D$ in $\Q^n.$  Guided by Jacquet's conjecture one could attempt to obtain general results of this form using the methods of this paper.  The Whittaker coefficients of the relevant metaplectic Eisenstein series have been studied in Chinta-Gunnells \cite{cg-inv} and have been shown to be Dirichlet series built out of Dirichlet $L$-functions.
\end{itemize}

\bibliographystyle{amsplain}
\bibliography{sources}

\end{document}